\documentclass[11pt,onecolumn]{article}

\usepackage{geometry}
\usepackage{amsfonts,amsmath,amssymb,euscript,dsfont,mathtools,amsthm}
\usepackage{subfig}
\usepackage{natbib}

\newtheorem{theorem}{\bf Theorem}[section]
\newtheorem{proposition}{\bf Proposition}[section]
\newtheorem{corollary}{\bf Corollary}[section]
\newtheorem{remark}{\bf Remark}[section]

\newcommand{\bmat}{\left[ \begin{matrix}}
\newcommand{\emat}{\end{matrix} \right]}

\newcommand{\N}{\mathbb  N}
\newcommand{\R}{\mathbb  R}

\newcommand{\bC}{\mathbf C}

\newcommand{\bSigma}{\boldsymbol{\Sigma}}

\newcommand{\cK}{\mathcal K}
\newcommand{\cI}{\mathcal I}
\newcommand{\cN}{\mathcal N}

\newcommand{\cS}{\mathcal S}

\newcommand{\Ebb}{{\mathbb E}\,}

\begin{document}

\date{} 
\title{\LARGE \bf On the Maximum Entropy Property of the First--Order Stable Spline Kernel and its Implications} 


\newcommand{\footremember}[2]{%
   \footnote{#2}
    \newcounter{#1}
    \setcounter{#1}{\value{footnote}}%
}
\newcommand{\footrecall}[1]{%
    \footnotemark[\value{#1}]%
} 
\author{%
    Francesca P. Carli \footremember{CL}{Department of Electrical Engineering and Computer Science, 
		University of Li\`{e}ge, Belgium and Department of Engineering, University of Cambridge, United Kingdom.}
    \\ {\tt\small  fpc23@cam.ac.uk}.}%

\maketitle

\begin{abstract} 
A new nonparametric approach for system identification has been recently proposed where the impulse response is seen as the realization of a zero--mean Gaussian process whose covariance, 
the so--called stable spline kernel, guarantees
that the 
impulse response is almost surely stable. 
Maximum entropy properties of the stable spline kernel have been pointed out in the literature. 
In this paper we provide an independent proof that relies on the theory of matrix extension problems 
in the graphical model literature 
and leads to a closed form expression for the inverse of the first order stable spline kernel as well as to a 
new factorization 
in the form $UWU^\top$ with $U$ upper triangular and $W$ diagonal. 
Interestingly, all first--order stable spline kernels share the same factor $U$ and 
$W$ admits a closed form representation in terms of the kernel hyperparameter, making the factorization computationally inexpensive.  
Maximum likelihood properties of the stable spline kernel are also highlighted. 
These results can be applied both to improve the stability and to reduce the computational complexity associated with the computation of stable spline estimators.

%
\end{abstract}

\section{Introduction}\label{sec:Introduction}


Most of the currently used techniques for linear system identification 
relies on parametric prediction error methods (PEMs), \citep{Ljung1999, Soderstrom1989}. 
Here, finite--dimensional hypothesis spaces of different order,
such as ARX, ARMAX or Laguerre models, are first postulated. 
Then, the most adequate model order 
is selected trading--off between bias and variance to avoid overfitting. 
Model--order selection is usually performed by optimizing  
some penalized goodness--of--fit criteria, such as the Akaike information criterion (AIC) \citep{Akaike1974} or the Bayesian information criterion (BIC) \citep{Schwarz1978}, 
or via cross validation (CV) \citep{HastieTibshirani2008}. 
%
Statistical properties of prediction error methods are well understood under the assumption that the model class is fixed. 
Nevertheless, 
sample properties of PEM approaches equipped e.g. with AIC or CV can much depart from those predicted by standard (i.e. without model selection) statistical theory 
(\citep{PillonettoDeNicolao2010,PillonettoChiusoDeNicolao2011}). 

Motivated by these pitfalls, 
a new approach to system identification has been recently proposed where 
the system impulse response is seen as the realization of a zero--mean Gaussian process 
with a suitable covariance that depends on few hyperparameters, 
learnt from data via, e.g., marginal likelihood maximization. 
This procedure 
can be seen as the counterpart of model order selection in the
parametric paradigm and in many cases it has been proved to be more robust than AIC-type criteria and CV. 

In this scheme, quality of the estimates crucially depends on the covariance (kernel) of the Gaussian process. 
A large variety of positive semidefinite kernels have been introduced in the machine learning literature 
\citep{ShaweTaylorCristianini2004,ScholkopfSmola2001}. 
Nevertheless, a straight application of standard machine learning kernels in the framework of system identification 
is doomed to fail 
mainly because of the lack of constraints on system stability. 
For this reason, several kernels have been recently introduced in the system identification literature \citep{PillonettoDeNicolao2010,ChenOhlssonGoodwinLjung2011}. 

This paper deals with \emph{stable spline kernels}.  
Stable spline kernels were introduced in \citep{PillonettoDeNicolao2010}  
%
%
as
an adaptation of spline kernels that 
enforces the associated process realizations to be asymptotically stable.  
Some theoretical results that assess robustness of this class of kernels 
are described in \citep{AravkinBurkeChiusoPillonetto2014,CarliChenChiusoLjungPillonetto2012}. 
Efficient numerical implementations 
are discussed in \citep{CarliChiusoPillonetto2012SYSID, ChenLjung2013}. 
In this paper we concentrate on \emph{first--order stable spline kernels} 
(see \citep{PillonettoChiusoDeNicolao2010_ACC} and also \citep{ChenOhlssonLjung2012}, 
where 
this class of kernels has also been introduced by using a totally different, deterministic argument).  
Maximum entropy properties of first--order stable spline kernels have been pointed out in \citep{PillonettoDeNicolao2011}. 
%
%
In this paper, we provide an alternative proof of the maximum entropy property by resorting to an independent, algebraic argument 
that connects to the theory of matrix completion and, in particular, of band extension problems in the graphical models literature 
\citep{Dempster1972,GroneJohnsonSaWolkowicz1984,DymGohberg1981,GohbergGoldbergKaashoek1993,DahlVanderbergheRoychowdhury2008}. 
%
%
This alternative approach leads to a 
closed form expression for the inverse of the first--order stable spline kernel. 
A factorization of the first--order stable spline kernel in the form $UWU^\top$ with $U$ upper triangular and $W$ diagonal is also provided. Interestingly, all first--order stable spline kernels share the same factor $U$ and 
$W$ admits a closed form representation in terms of the kernel hyperparameter, 
making the factorization inexpensive from a computational point of view.  
Moreover it can be proved that 
the first--order stable spline kernel 
maximizes the likelihood among all covariances that satisfy certain conditional independence constraints.  
The above mentioned properties can for example be used both to improve stability and reduce the computational burden  
of computational schemes for the evaluation of the stable spline estimator. 
 
The paper is organized as follows. 
In Section \ref{sec:LinearSYSDviaGP} the problem is introduced and Gaussian process regression via first order stable--spline kernels is briefly reviewed. 
In Section \ref{sec:MaxEntrBandExt} relevant theory of matrix completion problems is introduced. 
Section \ref{sec:FirstOdertSSk_as_MaxEntrkernel} contains our main results. 
Section \ref{sec:Conclusions} ends the paper.  

\vspace{4mm}
\noindent \textbf{Notation.} 
Let $\cS_n$ denote the vector space of symmetric matrices of order $n$. 
We write $A \succeq 0$ (resp. $A \succ 0$) to denote that $A$ is positive semidefinite (resp. positive definite). 
Moreover, we denote by $I_k$ the identity matrix of order $k$, 
by $\mathds{1}_k$ the $k$--dimensional vector of all ones,  
and by $0_k$ the $k$--dimensional vector of all zeroes. 
The diagonal matrix of order $k$  with diagonal elements $\left\{a_1, a_2, \dots, a_k\right\}$ 
will be denoted by ${\rm diag}\left\{a_1, a_2, \dots, a_k\right\}$. 
If $A$ is a square matrix of order $n$, for index sets $\beta \subseteq \left\{1, \dots, n\right\}$ and $\gamma \subseteq \left\{1, \dots, n\right\}$,  
we denote the submatrix that lies in the rows of $A$ indexed by $\beta$ and 
the columns indexed by $\gamma$ as $A(\beta,\gamma)$. 
If $\gamma=\beta$, the submatrix $A(\beta, \gamma)$ is abbreviated $A(\beta)$.

\section{Linear system identification via Gaussian Process Regression}\label{sec:LinearSYSDviaGP}

\subsection{Statement of the problem}\label{subsec:ProblemStatement}
We consider the measurement model 
\begin{equation}\label{eqn:lin_mod}
 y_t = \sum_{k=1}^{\infty} f_k u_{t-k} + e_{t} 
\end{equation}
where $\{y_t\}$ denote the noisy output samples of a discrete--time linear dynamical system 
fed with a known input $\{u_t\}$. 
$f=\{f_t\}_{t=1}^{\infty}$ is the unknown impulse response and $\{e_t\}$ is white Gaussian noise with 
variance $\sigma^2$.  
Suppose that $N$ measurements are available. 
We can collect these measurements in the $N$--dimensional column vector $y=[y_1,\, \ldots, \, y_N]^\top$. 
Let $e$ denote $N$--dimensional vector of the noise samples $e=[e_1 \quad \ldots e_N]^\top$. 
Thinking of $f$ as an infinite--dimensional column vector, and using notation of ordinary algebra to handle infinite--dimensional objects, model \eqref{eqn:lin_mod} can be expressed in matrix form as 
\begin{equation}\label{Uf}
y= Gf + e
\end{equation}
where $G \in \R^{N \times \infty}$ is a matrix whose entries are defined by the system input, 
so that $Gf$ represents  the convolution between the system impulse response and the input.  
We consider the problem of estimating $f$ from $y$.

\subsection{Gaussian process regression via Stable Spline Kernels}\label{subsec:GPregression_via_SSkernel}

In the classical system identification set up, the impulse response 
is searched for within a finite--dimensional space, e.g. postulating ARX, ARMAX or Laguerre models. 
Under the framework of Gaussian process regression \citep{RasmussenWilliams2006}, 
$f$ is instead modeled as a sampled version of a continuous--time zero--mean Gaussian process
with a suitable covariance (kernel), independent of $e$. 
We denote with $K$ the infinite--dimensional matrix obtained by sampling $K(\cdot,\cdot)$ on $\N \times \N$ 
and write 
\begin{equation}\label{modf}
f \sim \cN(0, K(\eta)), \qquad f \perp e
\end{equation}
where $\eta$ is a vector of hyperparameters governing the prior covariance. 
According to an Empirical Bayes paradigm \citep{Berger1985, MaritzLwin1989}, 
the hyperparameters can be estimated from the data via marginal likelihood maximization, i.e. 
by maximizing the marginalization with respect to $f$ of the joint density of $y$ and $f$ 
\begin{align}\label{eqn:J_N}
\hat \eta 
&= \arg\min_{\eta} \;\left\{\log \det \Sigma_y(\eta) +  y^\top \Sigma_y(\eta)^{-1} y\right\} 
\end{align}
with
\begin{equation}\label{eqn:Sigmay}
\Sigma_y(\eta) = G K(\eta) G^\top + \sigma^2 I_N \,. 
\end{equation}
Once $\eta$ is estimated, the impulse response can be computed as the minimum variance estimate 
given $y$ and $\hat{\eta}$, i.e. 
\begin{equation}\label{eqn:f}
\hat f := \Ebb \left[f | y, \hat \eta \right]  
= K(\hat \eta) G^\top \left(G K(\hat \eta) G^\top + \sigma^2 I_N\right)^{-1} y \,.
\end{equation}
Prior information is introduced in the identification process by assigning the covariance $K(\eta)$.  
The quality of the estimates 
crucially depends on this choice 
as well as on the quality of the estimated $\hat \eta$.  

A class of prior covariances which has been proved to be very effective in the system identification scenario, 
is the class of stable spline kernels (\citep{PillonettoDeNicolao2010,PillonettoChiusoDeNicolao2010_ACC,PillonettoChiusoDeNicolao2011}), that, besides incorporating information on smoothness, 
guarantees that the estimated impulse response is almost surely stable.

First--order stable spline kernels (equivalently, stable spline kernels of order $1$) 
were introduced in \citep{PillonettoChiusoDeNicolao2010_ACC} 
(see also \citep{ChenOhlssonLjung2012}, where they are referred to as Tuned/Correlated (TC) kernels)  
and are defined as 
$$
\cK_{ij} = \lambda \alpha^{\max(i,j)}, \qquad \lambda\geq 0,\; 0 \leq \alpha < 1, 
$$
so that $\eta = [\lambda, \alpha]$. 

\section{Maximum Entropy band extension problem}\label{sec:MaxEntrBandExt}

Covariance extension problems were introduced by A.~P. Dempster \citep{Dempster1972} and studied by many authors (see e.g. 
\citep{GroneJohnsonSaWolkowicz1984,DymGohberg1981,Johnson1990,GohbergGoldbergKaashoek1993,DahlVanderbergheRoychowdhury2008} and references therein, 
see also \citep{CFPP-2011,CG-2011,CFPP-2013} for an extension to the circulant case). 
In the literature concerning matrix completion problems,     
it is common practice to describe the pattern of the specified entries of an $n\times n$ partial symmetric matrix 
by an undirected graph 
of $n$ vertices which has an edge joining vertex $i$ and vertex $j$ if and only if the $(i,j)$ entry 
is specified. 
If the graph of the specified entries is \emph{chordal} (i.e., a graph in which every cycle of length greater than three has an edge connecting nonconsecutive nodes, see e.g. \citep{Golumbic-80}), and, in particular, if the specified elements lie on a band centered along the main diagonal, then 
the maximum entropy covariance extension problem admits a closed form solution in terms of the principal minors of the matrix to be completed 
(see \citep{BJL-89}, \citep{FKMN-00}, \citep{NFFKM-03}).  
In this section, we briefly review some fundamental results 
about \emph{maximum entropy band extension problems} that 
will be used 
to prove our main results in Section \ref{sec:FirstOdertSSk_as_MaxEntrkernel}. 

Recall that the {\em differential entropy} $H(p)$ of a probability density function $p$ on $\R^n$ is defined by
\begin{equation}\label{DiffEntropy}
H(p)=-\int_{\R^n}\log (p(x))p(x)dx.
\end{equation}
In  case of a zero--mean Gaussian distribution $p$ with covariance matrix $\bSigma_n$, we get
\begin{equation}\label{gaussianentropy}
H(p)=\frac{1}{2}\log(\det\bSigma_n)+\frac{1}{2}n\left(1+\log(2\pi)\right).
\end{equation}
Let $\cI \subset \left\{1,\dots ,n\right\}\times \left\{1,\dots,n\right\}$ denote a set of indices  
and $\bar{\cI}$ the complement of $\cI$ with respect to $\left\{1,\dots ,n\right\}\times \left\{1,\dots,n\right\}$. 
Let $x$ be the vector, say $k$--dimensional, obtained by stacking the $x_{ij}$'s one on top of the other. 
A \emph{partial matrix} is a parametric family of $n \times n$ matrices $\bSigma_n(x)$ 
with entries $[\bSigma_n(x)]_{i,j}=\sigma_{ij}$, $(i,j) \in \cI$ specified, 
and entries $[\bSigma_n(x)]_{i,j} = x_{ij}$, for $(i,j) \in \bar{\cI}$, which are left unspecified. 
Here, both $\sigma_{ij}$ and $x_{ij}$ are taken to be real.
A \emph{completion (extension)} of $\bSigma_n(x)$ is a $n \times n$ matrix $[C]_{i,j} = c_{ij}$ which satisfies 
$$c_{ij} = \sigma_{ij}\quad  \forall (i,j) \in \cI\,.$$  
{In particular, let } 
$$
\cI^{(m)}_b:= \left\{(i,j)\, \mid \, |i-j|\leq m \right\}\,. 
$$
If $\cI \equiv \cI^{(m)}_b$, we refer to $\bSigma_n^{(m)}(x)$ as a {\emph{partially specified $m$--band matrix}.  

Consider the following optimization problem 
\begin{subequations}\label{probl:MaxEntr}
\begin{eqnarray}
\underset{}{{\rm minimize}} & \left\{-\log \det \bSigma_n^{(m)}(x) \mid \bSigma_n^{(m)}(x) \in \cS_n\right\}\\ 
\text{subject to } & \bSigma_n^{(m)}(x) \succeq 0\\
& e_i^\top \, \bSigma_n^{(m)}(x)\, e_j = \sigma_{ij}, \quad (i,j)\in \cI^{(m)}_b 
\end{eqnarray}
\end{subequations}
with optimization variable $x$, namely the problem of computing the maximum entropy extension  
of the partially specified symmetric $m$--band matrix $\bSigma_n^{(m)}(x)$.     
Problem \eqref{probl:MaxEntr} is a convex optimization problem. 
Denote by $x^o$ its optimal value and by $\bSigma_n^{(m),o}\equiv\bSigma_n^{(m)}(x^o)$ the associated extension.  
Moreover from now on, we will drop the dependence on $x$ in $\bSigma_n^{(m)}(x)$ and refer to  
a $m$-band partially specified $n \times n$ matrix as $\bSigma_n^{(m)}$.  

\begin{theorem}[\citep{Dempster1972, DymGohberg1981}] 
\label{thm:feasibility_and_bandedness}
\begin{itemize}
	\item[$(i)$] \emph{Feasibility:} Problem \eqref{probl:MaxEntr} is feasible, 
	namely $\bSigma_n^{(m)}$ admits a positive definite extension if and only if 
	\begin{equation}\label{eqn:CNS_A_admits_posdef_ext}
	\begin{bmatrix}
	\sigma_{i,i} & \cdots & \sigma_{i,m+i} \\
	\vdots & & \vdots \\
	\sigma_{m+i,i} & \cdots & \sigma_{m+i,m+i}
	\end{bmatrix}
	\succ 0, 
	\;\,\,
	i=1, \ldots, n-m
	\end{equation}
	\item[$(ii)$] \emph{Bandedness: } 
	Assume \eqref{eqn:CNS_A_admits_posdef_ext} holds. 
	Then 	\eqref{probl:MaxEntr} admits a unique solution with the additional property that its inverse is banded 
	of bandwidth $m$, namely the $(i,j)$--th entry of $\left(\bSigma_n^{(m),o}\right)^{-1}$ 
	is zeros if $|i-j|>m$. 
\end{itemize}
\end{theorem}
The positive definite maximum entropy extension $\bSigma_n^{(m),o}$ 
is also called \emph{central extension} of $\bSigma_n^{(m)}$. 

Let $\bar \Sigma$ be such that $\left[\bar \Sigma\right]_{ij} = \sigma_{ij} \quad (i,j) \in \cI_b^{(m)}$. 
Then, it can be shown \citep{Dempster1972,DahlVanderbergheRoychowdhury2008} that Problem \eqref{probl:MaxEntr} 
is equivalent to the following optimization problem 
\begin{subequations}\label{probl:MaxLik} 
\begin{eqnarray}
\underset{ }{{\rm minimize}} & \log \det \bSigma_n^{(m)} 
+ {\rm trace} \left(\bar \Sigma  \left(\bSigma_n^{(m)}\right)^{-1} \right) \label{eqn:maxlik_obj}\\ 
\text{subject to } & \bSigma_n^{(m)} \succeq 0 \label{eqn:maxlik_c1}\\
& e_i^\top \, \left(\bSigma_n^{(m)}\right)^{-1}\, e_j = 0, \quad (i,j)\in \bar{\cI}^{(m)}_b   \label{eqn:maxlik_c2}
\end{eqnarray}
\end{subequations}
If we denote with $\theta = \bmat \theta_1, \dots, \theta_n\emat^\top$ a zero--mean Gaussian random vector with covariance $\bSigma_n^{(m)}$, 
then \eqref{eqn:maxlik_c2} holds if and only if the random variables $\theta_i$, $\theta_j$ in $\theta$ 
 are conditionally independent  given the others (see e.g. \citep{Dempster1972}). 
In other words, if we denote with $\bar \bSigma$ the sample covariance of $\theta$, 
the equivalence between Problem \eqref{probl:MaxEntr} and Problem \eqref{probl:MaxLik} 
states that the 
covariance matrix that maximizes the entropy among all the covariance matrices with 
given first $m+1$ covariance lags, is also the one that maximizes the likelihood among all 
the covariance matrices satisfying the conditional independence constraints \eqref{eqn:maxlik_c2}. 

For banded sparsity pattern like those considered so far, Problem \eqref{probl:MaxEntr} admits a closed form solution that can be computed recursively in the following way. 
We start by considering a partially specified $n \times n$ symmetric matrix of bandwidth $(n-2)$  
\begin{equation}\label{eqn:def_Sigma0}
\bSigma_{n}^{(n-2)} =
\bmat
\sigma_{1,1} & \sigma_{1,2} & \ldots & \sigma_{1,n-1} & x \\
\sigma_{1,2} & \sigma_{2,2} & \ldots & \sigma_{2,n-1} & \sigma_{2,n} \\
\vdots & \vdots  & & \vdots & \vdots \\
\sigma_{1,n-1} & \sigma_{2,n-1} & \ldots & \sigma_{n-1,n-1} & \sigma_{n-1,n} \\
x & \sigma_{2,n} & \ldots & \sigma_{n-1,n} & \sigma_{n,n}\\
\emat
\end{equation}
and consider the submatrix 
\begin{equation}\label{eqn:LMR}
L = [\sigma_{ij}]_{i,j = 1}^{n-1}\,. 
\end{equation}
We call \emph{one--step extensions} the extensions of $n \times n$, $(n-2)$--band matrices. 
The following theorem gives a recursive algorithm to compute the extension 
of partially specified matrices of generic bandwidth $m$ by computing the one--step extensions 
of suitable submatrices. It also gives a representation of the solution in factored form. 

\begin{theorem}[\citep{GohbergGoldbergKaashoek1993}, \citep{DymGohberg1981}]
\label{thm:forma_chiusa_estensione_centrale}
\begin{itemize}
	\item[$(i)$] The one--step central extension of $\bSigma_{n}^{(n-2)}$ is given by
 \begin{equation}\label{eqn:z0}
x^o = - \frac{1}{y_1} \sum_{j=2}^{n-1} \sigma_{nj} y_j
\end{equation}
with 
\begin{equation}\label{eqn:y}
\bmat  y_1\\ y_2 \\ \vdots \\ y_{n-1} \\ \emat
= L^{-1} \bmat 1 \\ 0 \\ \vdots \\ 0 \\ \emat \,. 
\end{equation}
Let $\bSigma_n^{(m)}$ be an $n \times n$ partially specified $m$--band matrix. 
The central extension $C = [c_{ij}]_{i,j=1}^{n} $ of $\bSigma_n^{(m)}$ is such that  
for all $m+1 \, < \, t \leq n$ 
and $1 \leq s \leq t-m-1$ the submatrices 
\begin{equation}\label{eqn:prop_estens_centrale}
C(\left\{s,\dots,t\right\}) = 
\begin{bmatrix}
c_{s,s} & \cdots & c_{s,t} \\
\vdots & & \vdots \\
c_{t,s} & \cdots & c_{t,t}
\end{bmatrix}
\end{equation}
are the central one--step extensions of the corresponding $(t-s-1)$--band matrix.  

	\item[$(ii)$] In particular, the central extension of the partially specified symmetric $m$-band matrix $\bSigma_n^{(m)}$ admits the factorization
\begin{equation}\label{eqn:central_extension_factored_form}
\bC = \left(L_n^{(m)}V U_n^{(m)}\right)^{-1}
\end{equation}
where $L_n^{(m)} = \left[\ell_{ij}\right]$ is a lower triangular banded matrix with ones on the main diagonal, 
$\ell_{jj} = 1$, for $j=1, \dots, n$, 
and
\begin{equation}\label{eqn:Xm}
\bmat \ell_{\alpha j}\\ \vdots \\\ell_{\beta j} \emat = 
- \bmat \sigma_{\alpha \alpha} & \dots & \sigma_{\alpha \beta}\\
\vdots & \ddots  &  \vdots\\
\sigma_{\beta \alpha} & \dots & \sigma_{\beta \beta}
\emat^{-1}
\bmat \sigma_{\alpha j}\\ \vdots \\\sigma_{\beta j} \emat 
\end{equation}
for $j=1, \dots, n-1$, $U_n^{(m)} = \left(L_n^{(m)} \right)^\top$ 
and $V=\left[v_{ij}\right]$  diagonal with entries 
\begin{equation}\label{eqn:V}
v_{jj} = \left(\bmat \sigma_{jj} & \dots & \sigma_{j \beta}\\
\vdots & \ddots  &  \vdots\\
\sigma_{\beta j} & \dots & \sigma_{\beta \beta}
\emat^{-1}\right)_{1,1}\,, 
\end{equation}
for $j=1, \dots, n$, where
$$
\begin{array}{lll}
\alpha = \alpha(j) = j+1 & \text{ for } & j=1, \dots, n-1\,,\\
\beta = \beta(j) = \min(j+m,n) & \text{ for } & j = 1, \dots, n\,. 
\end{array}
$$
\end{itemize}
\end{theorem}



\section{Maximum Entropy properties of the First-order Stable Spline kernel and Its Implications}\label{sec:FirstOdertSSk_as_MaxEntrkernel}

%
%
%

In this section, we provide an independent proof of the maximum entropy property of first--order stable spline kernels that relies on the theory of matrix extension problems 
introduced in the previous section. This argument leads to a closed form expression for the inverse of the first order stable spline kernel as well as to a new factorization. 
Maximum likelihood properties of the stable spline kernel are also highlighted.

	\begin{proposition}\label{prop:MaxEntr_for_TCkernel}
	Consider Problem \eqref{probl:MaxEntr} with $m=1$ and 
	\begin{equation}\label{eqn:moment_contraints_TC}
	\sigma_{ij} = \cK_{ij} = \alpha^{\max(i,j)}, \quad (i,j) \in \cI_b^{(1)}
	\end{equation}
	i.e. consider the partially specified $1$--band matrix
	$$
	\bSigma_n^{(1)}(x) = \bmat
	\alpha & \alpha^2 & x_{13} & \dots & \dots & x_{1n}\\
	\alpha^2 & \alpha^2 & \alpha^3 & x_{24}  & \dots & x_{2n}\\
	x_{13} & \alpha^3 & \alpha^3 & \alpha^4 &  & \vdots \\
	\vdots &  & \ddots & \ddots & \ddots & x_{n-2,1}\\
	\vdots &  &  & \alpha^{n-1}& \alpha^{n-1} & \alpha^{n}\\
	x_{1n} & \dots & \dots & x_{n-2,1} & \alpha^n & \alpha^n
	\emat
	$$
	Then  
	$\bSigma_n^{(1)}(x^o) = \cK$, i.e.  
	the solution of the Maximum Entropy Problem \eqref{probl:MaxEntr} coincides with the first order stable spline kernel. 
	\end{proposition}
	\vspace{3mm}

	\begin{proof} 
	By Theorem \ref{thm:forma_chiusa_estensione_centrale}, the maximum entropy completion of $\bSigma_n^{(1)}(x)$ can be recursively computed 
	starting from the maximum entropy completions of the nested principal submatrices of smaller size.  
	The statement can thus be proved by induction on the dimension $n$ of the completion. 
	\begin{itemize}
		\item Let $n=3$, then 
	by \eqref{eqn:z0}--\eqref{eqn:y}, the central extension of 
	$$
	\bmat \alpha & \alpha^2 & x_{13}\\
	\alpha^2 &\alpha^2 & \alpha^3 \\
	x_{13} & \alpha^3 &\alpha^3\emat 
	$$
	is given by $x_{13}^o= \alpha^3 = \cK(1,3)$, as claimed. 

	\item Now assume that the statement holds for $n=k$, $k \geq 3$, i.e. that 
	$\cK(\left\{1,\dots, k\right\})$ is the central extension of $\bSigma_n^{(1)}(\left\{1,\dots, k\right\})$. 
	We want to prove that $\cK(\left\{1,\dots, k+1\right\})$ 
	is the central extension of $\bSigma_{n}^{(1)}(\left\{1,\dots, k+1\right\})$.  
	To this aim, we only need to prove that the $(k-1)$ submatrices $\cK(\left\{s,\dots, k+1\right\})$, $1 \leq s \leq k-1$, 
 	are the central one--step extensions of the corresponding $(k-s)$--band matrices   
  $$
		\bmat
		\alpha^s & \dots & \alpha^k & x_{s,k+1}\\
		 \vdots & \ddots &  & \alpha^{k+1}\\
		 &  & \ddots & \vdots \\
		x_{s,k+1} & \alpha^{k+1} & \dots & \alpha^{k+1} 
		\emat\,.
		$$
	or, equivalently, that $x_{s,k+1}^o = \cK(s,k+1)$, for $s=1, \dots, k-1$. 
	In order to find $x_{s,k+1}^o$, we consider \eqref{eqn:y}, which, by the inductive hypothesis, becomes  
	\begin{equation*}
	\bmat  y_1^{(s,k+1)}\\ y_2^{(s,k+1)} \\ \vdots \\ y_{k}^{(s,k+1)} \\ \emat
	= \cK(\left\{s, \dots, k\right\})^{-1} \bmat 1 \\ 0 \\ \vdots \\ 0 \\ \emat \,. 
	\end{equation*}
	By considering the adjoint of $\cK(\left\{s, \dots, k\right\})$ 
	one can see that 
	$y_2^{(s,k+1)} = -  y_1^{(s,k+1)}$ 
	while all the others $y_{i}^{(s,k+1)}$, $i=3, \dots, k$ are identically zero. 
	It follows that 
	$$
	x_{s,k+1}^o = -\frac{1}{y_1^{(s,k+1)}} \, y_2^{(s,k+1)} \alpha^{k+1}  = \alpha^{k+1}	= \cK(s,k+1)\,,
	$$
	as claimed. 
	
\end{itemize}
	
	\end{proof}

From the equivalence between the maximum entropy problem \ref{probl:MaxEntr} and 
the maximum likelihood problem \ref{probl:MaxLik} we get the following. 
\begin{proposition}
Let $m=1$ and 
$\left[\bar \Sigma\right]_{ij}= \left[\theta \theta^\top\right]_{ij}=\alpha^{\max(i,j)}$, $(i,j) \in \cI_b^{(1)}$,  
then the first-order stable spline kernel maximizes the likelihood in \eqref{eqn:maxlik_obj} among all covariances that satisfies \eqref{eqn:maxlik_c1} and the conditional independence constraints \eqref{eqn:maxlik_c2}. 
\end{proposition}


\begin{proposition}
The following are equivalent
\begin{itemize}
	\item[$(i)$] $\cK$ solves Problem \eqref{probl:MaxEntr} with $m=1$ and 
	moment constraints as in \eqref{eqn:moment_contraints_TC}. 
	\item[$(ii)$] $\cK$ admits the factorization 
\begin{equation}\label{eqn:fatt_KTC}
\cK = U W U^\top
\end{equation}
with 
\begin{equation}\label{eqn:Lm_invT_TCkernel}
U = \bmat 1&1& \dots &1\\0&1&\dots&1\\0&0& \ddots&\vdots\\0&\dots & 0 & 1\emat\,,
\end{equation}
and 
\begin{equation}\label{eqn:Vinv_TCkernel}
W = {(\alpha-\alpha^2)}{\rm diag}\left\{1, {\alpha},{\alpha^2}, \dots, 
{\alpha^{n-2}},\frac{\alpha^{n-1}}{1-\alpha}\right\}\,.
\end{equation}
\end{itemize}
\item[$(iii)$] 
$\cK^{-1}$ is tridiagonal banded and is given by 
\begin{equation}\label{eqn:KTC_inv}
\frac{1}{\alpha-\alpha^2}\bmat
1&-1&0&\dots&0\\
-1&1+\frac{1}{\alpha} & -\frac{1}{\alpha} & \ddots & \vdots\\
0&  &\frac{1}{\alpha}+\frac{1}{\alpha^2} &  & 0\\
\vdots&\ddots&  & \ddots& -\frac{1}{\alpha^{n-2}}\\
0& \dots& 0& -\frac{1}{\alpha^{n-2}}& \frac{1}{\alpha^{n-2}}+\frac{1-\alpha}{\alpha^{n-1}}
\emat
\end{equation} 
\end{proposition}

\begin{proof} 
That $\cK$ admits the factorization \eqref{eqn:fatt_KTC}--\eqref{eqn:Vinv_TCkernel} follows from 
Theorem \ref{thm:forma_chiusa_estensione_centrale} (ii). In fact, by \eqref{eqn:central_extension_factored_form}--\eqref{eqn:V}  
the inverse of the stable spline kernel of order $1$ can be factored as  
\begin{equation}\label{eqn:sol_MaxEnt_per_TCkernel}
\cK^{-1} = L_n^{(1)}V U_n^{(1)}
\end{equation}
where 
$L_n^{(1)}$ takes the form 
\begin{equation}\label{eqn:Lm_TCkernel}
L_n^{(1)} = \left(U_n^{(1)}\right)^\top = \bmat
1 & 0 & 0 & \dots & 0\\
-1 & 1 & 0 & \dots & 0\\
0 & \ddots & \ddots & \ddots & \vdots\\
\vdots & \ddots & \ddots &  \ddots & 0\\
0 & \dots & 0 & -1 & 1
\emat
\end{equation}
and
\begin{equation}\label{eqn:V_TCkernel}
V = \frac{1}{\alpha-\alpha^2}{\rm diag}\left\{1, \frac{1}{\alpha},\frac{1}{\alpha^2}, \dots, 
\frac{1}{\alpha^{n-2}},\frac{1-\alpha}{\alpha^{n-1}}\right\}
\end{equation}
Bandedness of $\cK^{-1}$ follows from Theorem \ref{thm:feasibility_and_bandedness} (ii) and expression \eqref{eqn:KTC_inv} 
for $\cK^{-1}$ is an immediate consequence of the factorization \eqref{eqn:fatt_KTC}--\eqref{eqn:Vinv_TCkernel}. 
\end{proof}


\begin{remark} 
Let $A=\bmat I_{n-1}\;\mid \;0_{n-1}\emat$.  It is an immediate consequence of \eqref{eqn:KTC_inv} that 
\begin{equation}\label{eqn:colonne_KTC_inv_sommano_1}
\left(A \cK^{-1}\right) \mathds{1}_{n-1} = 0_{n-1}
\end{equation}
i.e. the first $n-1$ columns of $\cK^{-1}$ sum up to zero. 
\end{remark}

\begin{corollary} The stable spline kernel of order $1$ has determinant   
\begin{equation}\label{eqn:det_K_TC}
{\rm det}\left(\cK\right) = \left[(1-\alpha)^{n-1} \alpha^{\frac12 n(n+1)}\right]\,.
\end{equation}
\end{corollary} 
\vspace{1mm}
\begin{proof} The first (resp., third) factor in the right hand side of \eqref{eqn:sol_MaxEnt_per_TCkernel} is a lower (resp., upper) triangular matrix with diagonal entries equal to one, and hence the positive definite matrix 
$\cK^{-1}$ and $V$ have the same determinant, i.e. 
$$
\det(\cK^{-1}) = \det(V) = 
\frac{1}{\alpha^n(\alpha-\alpha^2)^{n-1}}\prod_{i=2}^n\frac{1}{\alpha^{i-2}}\,.
$$
The thesis follows immediately by recalling that $\sum_{i=1}^{n-2} i = \frac12 (n-2)(n-1)$\,. 
\end{proof}

\begin{remark} A key point in the evaluation of the stable spline estimator lies in solving the marginal likelihood maximization problem \eqref{eqn:J_N}, 
that is usually nonconvex. 
No matter what solver is used, the tuning of the hyperparameters requires repeated evaluations of the marginal likelihood.  
Here we observe that, whatever the value of $\alpha$, all the stable spline kernels of order $1$ share the same factor $U$ \eqref{eqn:Lm_invT_TCkernel}.  
Moreover, being $W$ available in closed form, once $\alpha$ is known the factorization \eqref{eqn:fatt_KTC} is computationally inexpensive.  
The same applies to the factorization of $\cK^{-1}$. 
This fact, together with the closed form expression for the determinant of the stable spline kernel in \eqref{eqn:det_K_TC},  
can be exploited both to improve the stability and to reduce the computational burden 
associated with computational schemes for the evaluation of the stable spline estimator like those in 
\citep{CarliChiusoPillonetto2012SYSID,ChenLjung2013}. 
\end{remark}

We conclude this section by highlighting 
an additional property 
of the first--order stable spline kernel 
that 
originates from the maximum entropy property 
of Proposition \ref{prop:MaxEntr_for_TCkernel}. 

\section{Conclusions}\label{sec:Conclusions}

Empirical Bayes estimation for system identification problems
has recently become popular, mainly due to the introduction
of a family of prior descriptions (the so--called stable spline kernels) which encode structural properties of dynamical systems such as stability. 
Maximum entropy properties of first--order stable spline kernels have been highlighted in 
\citep{PillonettoDeNicolao2011}. 
In this paper we provide an alternative proof that leads 
to a closed form expression for the inverse of the first order stable spline kernel as well as to a 
new, computationally advantageous factorization. 
Maximum likelihood properties of the stable spline kernel are also highlighted. 
These properties can be exploited both to improve the stability and to relieve the computational complexity associated with the computation of 
stable spline estimators. 



%
%
%
%
%
%
%
%

\bibliographystyle{plainnat}
\bibliography{biblio_TC_kernel}

\end{document}